\documentclass[12pt]{amsart}
\usepackage{amssymb,amsmath,amsthm}
\usepackage{mathtools,enumerate}

\newtheorem{theorem}{Theorem}[section]

\newtheorem{lem}{Lemma}[section]

\DeclarePairedDelimiter\floor{\lfloor}{\rfloor}

\newcommand{\N}{\mathbb{N}}
\newcommand{\Z}{\mathbb{Z}}

\newcommand{\Q}{\mathbb{Q}}

\title[ Reduction of polynomial dynamical systems modulo primes]{Reduction of polynomial dynamical systems modulo primes }

\author[S. S. Rout]{S. S. Rout}

\thanks{}



\address{Sudhansu Sekhar Rout
\hfill\break\indent Institute of Mathematics \& Applications,
\hfill\break\indent Andharua, Bhubaneswar-751 029,
\hfill\break\indent Odisha, India}

\email{lbs.sudhansu@gmail.com}

\begin{document}

\begin{abstract}
We study the algebraic dynamical systems generated by triangular systems of rational functions and estimate the  height growth of iterations generated by such systems. Further, using a result on the reduction modulo primes of systems of multivariate polynomials over the integers, we study the periodic points and the intersection of orbits of such dynamical systems over finite fields. 
\end{abstract}

\maketitle

\section{Introduction}

Let $V \subset \mathbb{P}^{N}$ be a quasi-projective variety defined over a field $K$ and let 
\[\Phi: V \longrightarrow V\]
be an endomorphism. For any $m \in \mathbb{N}_{0} = \mathbb{N} \cup \{0\}$, we denote by $\Phi^{(m)} = \Phi \circ \cdots \circ \Phi$ the $m$-th iteration of $\Phi$ with $\Phi^{(0)}$ denoting the identity map. For a given point $P\in V(K)$, the (forward) orbit of $P$ is the set
\[\mbox{Orb}_{\Phi}(P) = \{P, \Phi^{(1)}(P),\Phi^{(2)}(P), \ldots\}.\]
The point $P$ is called a {\em periodic point} for $\Phi$ if $\Phi^{(n)}(P) = P$ for some $n\geq 1$ and the smallest such $n$ is called the {\em period} of $P$. The point $P$ is called {\em preperiodic} if some iterate $\Phi^{(m)}(P)$ is periodic.

The area of algebraic dynamics was introduced by Northcott \cite{nor} and later, Silverman  \cite{sil} greatly developed all aspects of the theory of algebraic dynamics. For a background of the dynamical systems associated with iterations, one can refer to \cite{sch,sil}. In \cite{sil08}, Silverman studied the orbit length for the reduction modulo a prime $p$ for any self morphism of a quasi-projective variety defined over a number field. Later, this result has been improved in \cite{ag}. Since then there have been many advances in the study of periodic points and period lengths in the reductions of orbits of dynamical systems modulo distinct primes $p$. 

Motivated by the work of Towsley \cite{tow13} on Hasse principle for periodic points, D'Andrea et. al., \cite{dosm}, using several tools from arithmetic geometry, have proved new results about the orbits of the reductions modulo a prime $p$ of algebraic dynamical systems over $\Q$. Later in \cite{cdosm}, Changa et. al., gives a lower bound for the orbit length of the reduction modulo primes of parametric polynomial dynamical systems defined over integers. As a by-product, their result recovers a result in \cite{ag} and slightly improves a result in \cite{sil08}. 

The results in \cite{dosm} depends on the growth of the degree and the height of the iterates. When this growth is slower than generic, one can expect stronger bounds. Although for a typical system an exponential degree growth is expected, there are rich families of multivariate polynomial systems with a much slower degree growth (see \cite{gos,hp, os, os12}). For example, for triangular system of polynomials, it has been shown in \cite{os} that degrees of the iterations of the polynomials in triangular system grow very slowly. 

In this paper, we consider the following class of rational dynamical systems with slow degree growth.  

Let 
\begin{equation}\label{eq1}
{\bf F} = (F_1, \ldots, F_n), \quad F_1, \ldots, F_n \in \mathbb{Q}({\bf X})
\end{equation}
 be a system of $n$ rational functions in $n$ variables $(X_1, \ldots, X_n)$ over $\mathbb{Q}$ where
\begin{align}\label{eq4}
\begin{split}
&F_{1}(X_1, \ldots, X_n) = X_1^{e_1} G_1(X_2, \ldots, X_n) + H_1(X_2, \ldots, X_n)\\
&\cdots \\
& F_{n-1}(X_1, \ldots, X_n) = X_{n-1}^{e_{n-1}}G_{n-1}( X_n) + H_{n-1}( X_n)\\
& F_{n}(X_1, \ldots, X_n) =  g_n X_n^{e_n} + h_n,
\end{split}
\end{align}
with $e_1, \ldots, e_n \in\{-1, 1\}, G_i, H_i \in \mathbb{Z}[X_{i +1}, \ldots, X_n], i= 1, \ldots, n-1$ and $g_n, h_n \in \mathbb{Z}, g_n\neq 0$.
 We define the iterations of the rational function $F_i$ as follows. 
\begin{align*}
&G_i^{(\ell)}(X_{i+1}, \ldots, X_n) = G_i(F_{i +1}^{(\ell-1)}, \ldots, F_n^{(\ell-1)}),\\
&H_i^{(\ell)}(X_{i+1}, \ldots, X_n) = H_i(F_{i +1}^{(\ell-1)}, \ldots, F_n^{(\ell-1)}).
\end{align*}

Our first main result in this paper  gives a bound for the number of points of a given period in the reduction modulo $p$ of the algebraic dynamical system defined in \eqref{eq4}.  Also, we give a bound for the frequency of the points in an orbit of the reduction modulo $p$ of the algebraic dynamical systems defined in \eqref{eq4} lying in a given algebraic variety. To prove these results, we use a deep result from arithmetic geometry \cite[Theorem 2.1]{dosm}.

\section{Notation and Main Results}
Let $\mathbf{X}$ denotes the group of variables $(X_1, \ldots, X_n)$ so that $\mathbb{Z}[{\bf X}]$ denotes the ring of polynomials $\mathbb{Z}[X_1, \ldots, X_n]$ and $\mathbb{Q}({\bf X})$ denotes the field of rational functions $\mathbb{Q}(X_1, \ldots, X_n)$.  Let $\overline{K}$ denote an algebraic closure of a field $K$ of characteristic zero.

For a polynomial $L\in \mathbb{Z}[{\bf X}]$, we define its height as the logarithm of the maximum of the absolute values of its coefficients and denote it by $h(L)$. For a rational function $F \in \mathbb{Q}({\bf X})$, we write $F= L/K$ with coprime $L, K \in \mathbb{Z}[{\bf X}]$ and we define the degree and the height of $F$, respectively, as the maximum of the degrees and of the heights of $L$ and $K$, that is,
\[\deg F = \max\{\deg L, \deg K\} \quad \mbox{and}\quad h(F) = \max\{h(L), h(K)\}.\]
To give explicit formula for degree growth of the iterates of the system in \eqref{eq4}, we impose the following conditions on the degrees of the polynomials $G_i$ and $H_i$ for $i = 1, \ldots, n-1$ (see \cite{os, os12}).

If $e_i =1$, we assume that the polynomial $G_i$ has a unique leading monomial $X_{i+1}^{s_{i, i+1}}\cdots X_n^{s_{i, n}}$, that is 
\[G_i = g_iX_{i+1}^{s_{i, i+1}}\cdots X_n^{s_{i, n}} + \tilde{G_i},\]
where $g_i \in \mathbb{Z}\setminus \{0\}$ and $\tilde{G_i} \in \mathbb{Z}[X_{i+1}, \ldots, X_n]$ with
\begin{equation}\label{eq4a}
\deg_{X_j} \tilde{G_i} <s_{i,j}, \quad \deg_{X_j} H_i <s_{i,j}, \;\; j=i+1, \ldots, n.
\end{equation}

If $e_i =-1$, we assume that the polynomial $H_i$ has a unique leading monomial $X_{i+1}^{s_{i, i+1}}\cdots X_n^{s_{i, n}}$, that is 
\[H_i = h_iX_{i+1}^{s_{i, i+1}}\cdots X_n^{s_{i, n}} + \tilde{H_i},\]
where $h_i \in \mathbb{Z}\setminus \{0\}$ and $\tilde{H_i} \in \mathbb{Z}[X_{i+1}, \ldots, X_n]$ with
\begin{equation}\label{eq4b}
\deg_{X_j} \tilde{H_i} <s_{i,j}, \quad \deg_{X_j} G_i < 2s_{i,j}, \;\; j=i+1, \ldots, n.
\end{equation}

We define the orbit of a given point ${\bf w} \in \overline{\mathbb{Q}}^n$ with respect to the system of rational functions in \eqref{eq1} as the set 
\begin{equation}\label{eq4c}
\mbox{Orb}_{{\bf F}}({\bf w}) = \{{\bf w}_{k} \;|\; {\bf w}_0 = {\bf w} \quad \mbox{and} \quad {\bf w}_{k} = {\bf F}({\bf w}_{k-1}), k = 1, 2, \ldots\}.
\end{equation}
If ${\bf w}_k$ is a pole of ${\bf F}$, then the orbit terminates and in this case $\mbox{Orb}_{{\bf F}}({\bf w})$ is a finite set. Further, given $k\geq 1$, we say that ${\bf w} \in \overline{\mathbb{Q}}^n$ is {\em $k$-periodic} if the element ${\bf w}_k$ exists in the orbit \eqref{eq4c} and we have ${\bf w}_k = {\bf w}_0$.  Further, we put 
\begin{equation}
S({\bf w}) = \# \mbox{Orb}_{{\bf F}}({\bf w}) \in \mathbb{N}\cup \{\infty\}.
\end{equation}

Let $p \in \mathbb{Z}$ be a prime.  For each prime $p$, set $F_{i,p}^{(m)} = F_{i}^{(m)}\, (\mbox{mod}\ p)$. We define the reduction modulo $p$ of the iteration ${\bf F}^{(m)}$ and denote this by 
 \[{\bf F}_p^{(m)} = (F_{1,p}^{(m)}, \ldots, F_{n,p}^{(m)}) \in \mathbb{F}_p[{\bf X}]^n.\]

Let ${\bf L} =(L_1, \ldots, L_s)\in \mathbb{Z}[{\bf X}]$  be a system of polynomials of degree at most $D$ and height at most $H$. We denote  by $V$ the subvariety of the affine space $\mathbb{A}_{\mathbb{Q}}^n$ defined by this system of polynomials. For a prime $p$, we denote by $L_{i,p}\in \mathbb{F}_p[{\bf X}]$ the reduction modulo $p$ of $L_i$ and by $V_p$ the subvariety of $\mathbb{A}_{\mathbb{F}_p}^n$ defined by the system $L_{i, p}, i=1, \cdots, s$. 

Given the functions
\[f, g : \mathbb{N} \longrightarrow \mathbb{R}\]
the symbols $f = O(g)$ and $f \ll g$ both mean that there is a constant $c\geq 0$ such that $|f(k)| \leq c g(k)$ for all $k \in \mathbb{N}$. To emphasize the dependence of the implied constant $c$ on a list of parameters, say $n, d, h$, we write $f = O_{n, d, h}(g)$ and $f\ll_{n, d, h} g$.

The following result is concerned with the number of points of a given period in the reduction modulo $p$ of triangular systems of polynomials as in \eqref{eq1} with $e_i=1$. 
\begin{theorem}\label{th1}
Let $F_1, \ldots, F_n \in \Z[{\bf X}]$ be as in \eqref{eq4} with $e_i=1$, satisfying the condition \eqref{eq4a}  such that $s_{i, i+1} \neq 0, i =1, \ldots, n-1$. Set 
\[d = \max_{j=1, \ldots, n} \deg F_j \quad \mbox{and}\quad h = \max_{j=1, \ldots, n} h( F_j).\]
Suppose that ${\bf F} = (F_1, \ldots, F_n)$ has finitely many periodic points of order $k$ over $\mathbb{C}$.   Then there exists an integer $\mathfrak{B}_1\in \mathbb{N}$ satisfying 
\[ \log \mathfrak{B}_1 \ll_{d, h, n} k^{n (3n -1)}\]

and such that if $p$ is a prime number not dividing $\mathfrak{B}_1$, then the reduction of ${\bf F}$ modulo $p$ has  $O_{d, h, n} (k^{n (n-1)/2})$ periodic points of order $k$.
\end{theorem}
In the following theorem, we study the same result as in Theorem \ref{th1} for system in \eqref{eq1} with $e_i = -1$.
\begin{theorem}\label{th1a}
Let $n\in \mathbb{N}$ with $n\geq 2$. For $i=1, \ldots, n$, let $F_i$ be rational functions defined by \eqref{eq4} satisfying the condition \eqref{eq4b} such that $s_{i, i+1} \neq 0, i =1, \ldots, n-1$ and $e_i=-1$ for $i=1, \ldots, n$. Set 
\[d = \max_{j=1, \ldots, n} \deg F_j \quad \mbox{and}\quad h = \max_{j=1, \ldots, n} h( F_j).\]
Suppose that ${\bf F} = (F_1, \ldots, F_n)$ has finitely many periodic points of order $k$ over $\mathbb{C}$.   Then there exists an integer $\mathfrak{B}_2\in \mathbb{N}$ satisfying 
\[ \log \mathfrak{B}_2 \ll_{d, h, n} k^{n(3n^2 +8n +9)/2}\]

and such that if $p$ is a prime number not dividing $\mathfrak{B}_2$, then the reduction of ${\bf F}$ modulo $p$ has  $O_{d, h, n} (k^{2n^2})$ periodic points of order $k$.
\end{theorem}

Next we obtain an upper bound for the frequency of the orbit intersections of a rational function system. More generally, we bound the number of points in such an orbit that belong to a given algebraic variety.

For $\ell \in \N$, let $p$ be a prime such that the iterations ${\bf F}^{(j)}$ with $j =0, \ldots, \ell-1$ can be reduced modulo $p$.  Given a point ${\bf w} \in \overline{\mathbb{F}}_p^n$, we define
\[\mathfrak{I}_w\left({\bf F}, V; p, \ell\right) = \left\{ m \in \{0, 1, \ldots, \ell-1\} \mid {\bf F}_p^{(m)}({\bf w}) \in V_p(\overline{\mathbb{F}}_p)\right\}.\]

We say that the iterations of ${\bf F}$ {\em generically escape} $V$ if for every integer $k\geq 1$, the $k$-th iteration of ${\bf F}$ is well defined and the set 
\[\left\{{\bf w} \in \mathbb{C}^n \mid \left ({\bf w}, {\bf F}^{(k)}({\bf w})\right) \in V(\mathbb{C})\times V(\mathbb{C})\right\}.\]
is finite.

\begin{theorem}\label{th2}
 Let $n\in \mathbb{N}$ with $n\geq 2$. For $i=1, \ldots, n$, let $F_i$ be rational functions defined by \eqref{eq4} satisfying the conditions \eqref{eq4a} and \eqref{eq4b} and such that $s_{i, i+1} \neq 0, i =1, \ldots, n-1$. Set 
\[d = \max_{j=1, \ldots, n} \deg F_j \quad \mbox{and}\quad h = \max_{j=1, \ldots, n} h( F_j).\]
Let $V$ be the subvariety of $\mathbb{A}_{\mathbb{Q}}^n$ defined by the system of polynomials $(L_1, \ldots, L_s)\in \mathbb{Z}[{\bf X}]$ of degree at most $D$ and height at most $H$. Assume that the iterations of ${\bf F}$ generically escape $V$. Then, there is a constant $C > 0$ (depending on D, H, d, h, n, s) such that for any real $\epsilon >0$ and $\ell \in \mathbb{N}$ with 
\[\ell \geq \frac{n^sD^{2s}}{\epsilon^{(n-1)s+2}}\]
there exists $\mathfrak{D}\in \mathbb{N}$ with 
\[\log \mathfrak{D} \leq C/ \epsilon^{n(3n-1)}\]
such that if $p$ is a prime number not dividing $\mathfrak{D}$, then for any ${\bf w} \in \overline{\mathbb{F}}_p^n$ with $S({\bf w}) \geq \ell$,
\[\# \mathfrak{I}_w\left({\bf F}, V; p, \ell\right) \leq \epsilon \ell.\]
\end{theorem}
Next we obtain a better result for the problem of bounding the frequency of the points in an orbit lying in a given variety under a restrictive condition.

Let ${\bf F}\in \mathbb{Q}[{\bf X}]^n$ be a system of rational functions over $K$ and let $V\subseteq \mathbb{A}_{\mathbb{Q}}^n$ be an affine variety. The intersection of the orbit  ${\bf F}$ with $V$ is {\em $L$-uniformly bounded} if there is a constant $L$ depending only on ${\bf F}$ and $V$ such that for all initial values ${\bf w}\in \overline{\mathbb{Q}}^n$,
\[\#\{m \in \mathbb{N}\mid {\bf w}_{m}\in V( \overline{\mathbb{Q}})\}\leq L,\]
where ${\bf w}_{m}$ is defined in \eqref{eq4c}.

\begin{theorem}\label{th3}
 Let $n\in \mathbb{N}$ with $n\geq 2$. For $i=1, \ldots, n$, let $F_i$ be rational functions defined by \eqref{eq4} satisfying the conditions \eqref{eq4a} and \eqref{eq4b} and such that $s_{i, i+1} \neq 0, i =1, \ldots, n-1$. Set 
\[d = \max_{j=1, \ldots, n} \deg F_j \quad \mbox{and}\quad h = \max_{j=1, \ldots, n} h( F_j).\]
Let $V$ be the subvariety of $\mathbb{A}_{\mathbb{Q}}^n$ defined by the system of polynomials $=(L_1, \ldots, L_s) \in \mathbb{Z}[{\bf X}]$ of degree at most $D$ and height at most $H$. Assume that the intersection of orbits of ${\bf F}$ with $V$ is $L$-uniformly bounded. There is a constant $C > 0$ (depending on D, H, d, h, n, L, s) such that for any real $\epsilon >0$ there exists $\mathfrak{C}\in \mathbb{N}$ with 
\[\log \mathfrak{C} \leq \frac{C}{ \epsilon^{(n-1)(3n+2)+(n+L+2)}}\]
such that if $p$ is a prime number not dividing $\mathfrak{C}$, then for any integer 
\[\ell \geq 2L/\epsilon+1\]
and for any initial point ${\bf w} \in \overline{\mathbb{F}}_p^n$ with $S({\bf w}) \geq \ell$, we have
\[\# \mathfrak{I}_w\left({\bf F}, V; p, \ell\right) \leq \epsilon\ell.\]
\end{theorem}

\section{Preliminaries}
In this section, we gather some bounds on the heights and the degrees of triangular polynomial systems. We start with bounds for the heights of sums and products of polynomials, which follows  from \cite[Lemma 1.2]{kps}.
\begin{lem}\label{lem1}
Let $K_1, \ldots, K_t \in \mathbb{Z}[{\bf X}]$. Then 
\begin{enumerate}[(1)]
\item $h\Big(\sum_{i =1}^t K_i  \Big) \leq \max_{1\leq i\leq t}h(K_i) +\log t$;\\
\item $-2\log (n+1)\sum_{i =1}^t \deg K_i \leq h\Bigg(\prod_{i=1}^t K_i\Bigg) - \sum_{i =1}^t h(K_i) \leq \log (n+1) \sum_{i =1}^t \deg K_i$.
\end{enumerate}
\end{lem}
The following is the standard bound for the degree and height of the composition of polynomials with integer coefficients (see \cite[Lemma 1.2(1.c)]{kps}).
\begin{lem}\label{lem2}
Let $L \in \mathbb{Z}[Y_1, \ldots, Y_t], K_1, \ldots, K_t \in \mathbb{Z}[{\bf X}]$. Set \[d = \max_{i =1, \ldots, t} \deg K_i\quad \mbox{ and} \quad h = \max_{i =1, \ldots, t} h( K_i).\] Then, 
\begin{align*}
&\deg (L(K_1, \ldots, K_t)) \leq d \deg L\\
&h(L(K_1, \ldots, K_t)) \leq h(L)+ \deg L(h + \log (t+1) + d \log (n+1)).
\end{align*}
\end{lem}

The following is an extension of Lemma \ref{lem2} to the composition of rational functions (see \cite{dosm}).
\begin{lem}\label{lem2d}
Let $L, K_1, \ldots, K_n \in \mathbb{Q}[{\bf X}]$ such that the composition $L(K_1, \ldots, K_n)$ is well defined.. Set \[d = \max_{i =1, \ldots, n} \deg K_i\quad \mbox{ and} \quad h = \max_{i =1, \ldots, n} h( K_i).\] Then, 
\begin{align*}
&\deg (L(K_1, \ldots, K_n)) \leq dn \deg L\\
&h(L(K_1, \ldots, K_n)) \leq h(L)+ h\deg L + (3dn +1)\deg L \log (n+1). 
\end{align*}
\end{lem}
The following lemma gives the degree growth of the iterations of function defined by \eqref{eq4}(see \cite[Theorem 2]{os12}). 

\begin{lem}\label{lem2a}
Let $F_1, \ldots, F_n$ be rational functions defined by \eqref{eq4} satisfying the conditions \eqref{eq4a} and \eqref{eq4b} and such that $s_{i, i+1} \neq 0, i =1, \ldots, n-1$. Then degrees of the iterations of $F_1, \ldots, F_n$ grow as follows
\begin{align*}
&\deg F_i^{(k)} = \frac{1}{(n-i)!}k^{n-i}s_{i, i+1}\cdots s_{n-1, n} + \psi_i(k), \quad i = 1, \ldots, n-1,\\
& \deg F_n^{(k)} =1
\end{align*}
where $\psi_i(T) \in \mathbb{Q}[T]$ is a polynomial of degree $\deg \psi_i < n-i$. 
\end{lem}

\begin{lem}\label{lem3}
For $i=1, \ldots, n$, let $G_i \in \mathbb{Z}[X_i, X_{i+1}, \ldots, X_n]$ be a triangular system of polynomials with a unique leading monomial of the form $X_{i+1}^{s_{i, i+1}}\cdots X_n^{s_{i, n}}$ and $F_i$ as in \eqref{eq4}.  
Set 
\[d = \max_{i=1, \ldots, n} \deg F_i, \;\mbox{and}\;\;  h = \max_{i =1, \ldots, n} h( F_i).\]
The height of the iterations of $G_1,\ldots, G_n$ for $k\geq 2$ grow as follow:
\begin{equation}\label{eq5t}
h( G_i^{(k)})\leq  \left(\sum_{j=1}^{k-1}\deg G_i^{(j)}\right) (h+\log (n-i)(n+1)^{d}) + h.
\end{equation}
Moreover, for any positive integer $k\geq 2$ and $1\leq i \leq n$, 
\begin{align}\label{al1}
\begin{split}
h( G_i^{(k)}) \ll_{h, d, n}k^{n-i+1}.
\end{split}
\end{align}
\end{lem}
\begin{proof}
 The inequality \eqref{eq5t} for the height follows by induction on the number of iterates $k$. We set for any $k\geq 1$ and $1\leq i \leq n$
 \[d_{i, k} =  \deg G_i^{(k)}, \quad h_{i, k} = h\left( G_i^{(k)}\right).\]For $k = 2$, we have $h(G_i^{(2)})= G_{i }(F_{i+1}^{(1)}, \ldots, F_{n}^{(1)})$. 
Now applying Lemma \ref{lem2} to this,
\begin{align*}
 h(G_i^{(2)}) &= h(G_{i }) + \deg(G_i)(h + \log (n-i) + d \log (n+1)) \\
 & = h_{i,1} + d_{i,1}(h + \log(n-i)+d\log(n+1))\\
 & \leq (d_{i,1} + 1)h + d_{i,1} \log(n-i)+ d d_{i,1}  \log(n+1).
\end{align*}
Thus, the inequality \eqref{eq5t} is true for $k=2$. Now assume that inequality \eqref{eq5t} is true for the first $k-1$ iterates. Applying Lemma \ref{lem2} to the polynomial 
\[G_i^{(k)} = G_{i }^{(k-1)}(F_{i+1}, \ldots, F_n),\]
\begin{align*}
h(G_i^{(k)})&= h(G_{i }^{(k-1)}) +\deg (G_{i }^{(k-1)}) (h +  \log (n-i) + d \log (n+1))\\
& \leq \left(\sum_{j=1}^{k-2}\deg G_i^{(j)}\right) (h+\log (n-i)(n+1)^{d}) + h\\
& +d_{i, k-1} (h +  \log (n-i) + d \log (n+1)).
\end{align*}
This proves inequality \eqref{eq5t}.
Now from \eqref{eq5t} and Lemma \ref{lem2a}, we have
\begin{align*}
&h( G_i^{(k)}) \leq  \left(\sum_{j=1}^{k-1}\deg G_i^{(j)}\right) (h+\log (n-i)(n+1)^{d}) + h\\
& = \left(\sum_{j=1}^{k-1}\deg G_i\left(F_{i+1}^{(j-1)}, \ldots, F_{n}^{(j-1)}\right)\right) (h+\log (n-i)(n+1)^{d}) + h\\
& = \left(\sum_{j=1}^{k-1}\deg \left((F_{i+1}^{(j-1)})^{s_{i,i+1}}\cdots (F_{n}^{(j-1)})^{s_{i,n}}\right)\right) (h+\log (n-i)(n+1)^{d}) + h\\
& = \left(\sum_{j=1}^{k-1}\left(\frac{1}{(n-i-1)!}(j-1)^{n-i-1}s_{i, i+1}\cdots s_{n-1,n}+\cdots + (j-1)s_{i,n}s_{n-1, n}+1\right)\right) \\
&\times (h+\log (n-i)(n+1)^{d}) + h \ll_{h, d, n}k^{n-i+1}.
\end{align*}
This completes the proof.
\end{proof}
Let us define the sets
\[I_{+} = \{1\leq i\leq n\mid e_i=1\}, \quad I_{-} = \{1\leq i\leq n\mid e_i=-1\}.\]

\begin{lem}\label{lem2b}
Let $F_1, \ldots, F_n$ be rational functions defined by \eqref{eq4} satisfying the conditions \eqref{eq4a} and \eqref{eq4b} and such that $s_{i, i+1} \neq 0, i =1, \ldots, n-1$. 
Then height of the iterations of $F_1, \ldots, F_n$ grow as follows:
\begin{align*}
&h(F_i^{(k)})  \leq (k+1)\deg \left( F_i^{(k)}\right) \log (n+1) + \sum_{j=1}^{k} h  \left(G_i^{(j)} \right)+ \log 2
\end{align*}
for every $i\in I_{+}$ and for every $i\in I_{-}$
\begin{align*}
h(F_i^k) \leq (k+1)\deg\left( F_{i}^{(k)}\right) \log(n+1)+ \sum_{j=1}^{k}h\left (H_{i}^{(j)}\right) + (k+1)\log 2.
\end{align*} 
Moreover, 
\[h \left( F_{i}^{(k)}\right) \ll_{d, h, n} k^{n-i+2}.\]
\end{lem}
\begin{proof}
First we prove the case when $i\in I_{+}$. The explicit structure of the iterations of the rational functions $F_{i}$ are given in \cite{os12}. By \cite[Lemma 2]{os12}, we have 
\begin{equation}\label{l3e1}
F_i^k = \begin{cases}
X_iG_{i,k} + H_{i,k}, & \text{for } i <n\\
g_n^kX_n + (g_n^{k-1} + \cdots + g_n +1)h_n & \text{for } i =n,
\end{cases}
\end{equation}
where \begin{align*}
&G_{i,k} = G_i G_i^{(2)}\cdots G_i^{(k)},\\
& H_{i,k} = H_i G_i^{(2)}\cdots G_i^{(k)} + H_{i}^{(2)} G_{i}^{(3)}\cdots G_{i}^{(k)} + \cdots + H_{i}^{(k-1)} G_{i}^{(k)} + H_i^{(k)}.
\end{align*}
Applying Lemma \ref{lem1} in equation \eqref{l3e1} for $i<n$, 
\begin{align*}
h(F_i^k) &\leq  h \left(X_iG_{i,k}\right) + \log 2 = h \left(X_iG_i G_i^{(2)}\cdots G_i^{(k)}\right) + \log 2 \\
& \leq \deg \left(X_iG_i G_i^{(2)}\cdots G_i^{(k)}\right) \log (n+1) + \sum_{j=1}^{k} h  \left(G_i^{(j)} \right)+h  \left(X_i\right)+ \log 2\\
& \leq (k+1)\deg \left( F_i^{(k)}\right) \log (n+1) + \sum_{j=1}^{k} h  \left(G_i^{(j)} \right)+\log 2. 
\end{align*}
Again using Lemma \ref{lem1} in \eqref{l3e1} for $i=n$, 
\begin{align*}
h(F_n^{(k)}) &= h\left( g_n^kX_n + (g_n^{k-1} + \cdots + g_n +1)h_n\right)\leq h( g_n^k X_n) + \log (k+1) \\
&\leq \log \left(g_n^k(n+1)(k+1)\right).
\end{align*}
Now consider the case $i\in I_{-}$ and $i<n$. In this case, by  \cite[Lemma 2]{os12}, we have 
\begin{equation}\label{l3e3}
F_i^{(k)} = \frac{X_i R_{i, k} + S_{i,k}}{X_i R_{i, k-1} + S_{i,k-1}},
\end{equation}
where $R_{i, k}, S_{i, k}$ are defined by the recurrence relations 
\begin{equation}\label{l3e4}
R_{i, k} = G_{i }^{(k)}R_{i, k-2} + H_{i}^{(k)}R_{i,k-1}, \quad S_{i, k} = G_{i }^{(k)}S_{i, k-2} + H_{i}^{(k)}S_{i,k-1}
\end{equation}
for $k\geq 1$ with the initial rational functions
\[R_{i, 0} =1, S_{i, 0} =0, R_{i, 1} =H_i, S_{i, 1} =G_i.\]
From Lemma \ref{lem1} and \eqref{l3e3}, 
\begin{align}\label{l3e5}
\begin{split}
h(F_i^k) &\leq  \max\{h \left(X_i R_{i, k} + S_{i,k}\right), h\left(X_i R_{i, k-1} + S_{i,k-1}\right)\} \\
& \leq h \left(X_i R_{i, k}\right) + \log 2\\
&\leq \deg  \left(X_i R_{i, k}\right) \log (n+1) + h(R_{i,k}) + h(X_i)+\log 2. 
\end{split}
\end{align}
Applying  Lemma \ref{lem1} in \eqref{l3e4}, one can inductively show that 
\begin{equation}\label{l3e6}
h(R_{i,k}) \leq (\deg (R_{i,0} \cdots R_{i,k})) \log(n+1)+ \sum_{j=1}^{k}h(H_{i}^{(j)}) + k\log 2
\end{equation}
Thus, for $i\in I_{-}$ and $i<n$,  from \eqref{l3e5} and \eqref{l3e6} we conclude
 \[ h(F_i^k) \leq (k+1)\deg\left( F_{i}^{(k)}\right) \log(n+1)+ \sum_{j=1}^{k}h\left (H_{i}^{(j)}\right) + h(X_i)+ (k+1)\log 2.\]

For the case $e_n=-1$, we have 
 \[F_n^{(k)} = \frac{(A^k)_{1, 1}X_n+ (A^k)_{1, 2}}{(A^k)_{2, 1}X_n+ (A^k)_{2, 2}},\, \; \mbox{where}\, A^k = \begin{pmatrix}
 h_n & g_n\\
 1 & 0\end{pmatrix}^k = \begin{pmatrix}
 (A^k)_{1, 1} & (A^k)_{1, 2}\\
(A^k)_{2, 1} & (A^k)_{2, 2}\end{pmatrix}.\]
 One can observe that the entries of the matrix $A^k$ are polynomials in the integer $h_n$ and $g_n$. Hence 
 \begin{align*}
 h(F_i^k) &\leq  \max\left\{h \left((A^k)_{1, 1}X_n+ (A^k)_{1, 2}\right), h\left((A^k)_{2, 1}X_n+ (A^k)_{2, 2}\right)\right\} \\
& \leq h \left((A^k)_{1, 1}X_n\right) + h((A^k)_{1, 2})+\log 2\\
&\leq  h \left((A^k)_{1, 1}\right) + h \left((A^k)_{1, 2}\right) + \log (2(n+1))\leq \log (h_n^k (k+1)(n+1)).
 \end{align*}
This completes the estimates of $h(F_i^k)$ for $i\in I_{-}$ and $i\leq n$.
 Also, 
\[\sum_{j=0}^{k}j^{n-i+1} = \frac{1}{n-i+2}(B_{n-i+2}(k+1)-B_{n-i+2}(0)),\]
where $B_{n-i+2}$ is the Bernoulli polynomial of degree $n-i+2$ with leading coefficient equal to $1$. 
Thus, from Lemma \ref{lem3}, the height of the $k$-th iteration of $F_{i}$ is at most
\[h \left( F_{i}^{(k)}\right) \ll_{d, h, n} k^{n-i+2}.\]
\end{proof}

The following result is on the reduction modulo primes of systems of multivariate polynomials over the integers, whose proof relies on the arithmetic Nullstenllensatz (see \cite[Theorem 2.1]{dosm}).
\begin{lem}\label{lem4}
Let $H_1, \ldots, H_s \in \mathbb{Z}[{\bf X}]$ be polynomials of degree at most $d \geq 2$ and height at most $h$, whose zero set in $\mathbb{C}^n$ has a finite number $T$ distinct points. Then there is an integer $\mathfrak{A}\geq 1$ with 
\[\log \mathfrak{A} \leq (11n +4)d^{3n+1}h + (55 n +99)\log ((2n+5)s) d^{3n+2}\]
such that if $p$ is a prime number not dividing $\mathfrak{A}$, then the zero set in $\overline{\mathbb{F}}_p^n$ of the system of polynomials $H_i  (\textrm{mod}\ p), i = 1, \ldots, s$ consists of exactly $T$ distinct  points.
\end{lem}

We also need the following combinatorial result \cite{dosm}.
\begin{lem}\label{lem4d}
Let $2\leq M \leq N/2$. For any sequence 
\[0\leq n_1< \cdots < n_M \leq N,\]
there exists  $r\leq 2N/(M-1)$ such that $n_{i+1} -n_i =r$ for at least $(M-1)^2/4N$ values of $i \in \{1, \ldots, M-1\}$.
\end{lem}

Now we are ready to proof our results. The proof is motivated by the ideas of  D'Andrea et. al., \cite{dosm}.

\section{Proof of Main Results}
\subsection{Proof of Theorem \ref{th1}}
Consider the system of equations 
\[F_i^{(k)} - X_i = 0, \quad i =1, \ldots, n.\]
The set of $k$-periodic points of ${\bf F}$ coincides with the zero set 
\[V_{k} = Z\left(F_1^{(k)} - X_1, \ldots, F_n^{(k)} - X_n\right).\]
For $i =1,\ldots, n$, 
\begin{equation}\label{eq10}
\deg\left( F_i^{(k)} - X_i \right)= \frac{1}{(n-i)!}k^{n-i}(s_{i, i+1}\cdots s_{n-1,n}) + 1,
\end{equation}
and
\begin{align}
\begin{split}\label{eq11}
&h\left( F_i^{(k)} - X_i \right) \leq \max\left\{ h\left( F_i^{(k)} \right), h(X_i) \right\}+ \log 2\ll_{d, h, n}k^{n-i+2}.
\end{split}
\end{align}
Now apply Lemma \ref{lem4} and derive
\begin{align*}
\log \frak{B}_1 &\leq C_1(n, d, h)(k^{n-1})^{3n+1}k^{n+1} + C_2(n, d, h)(k^{n-1})^{3n+2}\\
& \ll_{d, h, n}k^{n(3n-1)}. 
\end{align*}
Suppose $T_{k}$ is the number of points of $V_{k}$ over $\mathbb{C}$ and this equal to the number of periodic points of order $k$ of ${\bf F}$ over $\mathbb{C}$. By Bezout's theorem, 
\[T_{k} \leq 2\prod_{i=1}^{n} k^{n-i} \ll_{n, d, h} k^{n(n-1)/2}.\]
This completes the proof. \qed
\subsection{Proof of Theorem \ref{th1a}}

By equation \eqref{l3e3}, the iterates of the system of rational functions ${\bf F}$ is given by 
\[F_i^{(k)} = \frac{X_i R_{i, k} + S_{i,k}}{X_i R_{i, k-1} + S_{i,k-1}}=: \frac{\Gamma_{i, k}}{\Psi_{i,k}}\]
with $\Psi_{i,k}\neq 0$ and consider the system of equations
\[\Gamma_{i,k}- X_i\Psi_{i, k} = 0, \quad i=1, \ldots, n.\]
To extract the poles of $F_i^{(j)}, j\leq k$ from the solutions of the system, we introduce a new variable $X_0$. Now the set of $k$-periodic points of ${\bf F}$ coincides with the zero set 
\[V_k = Z\left(\Gamma_{1,k}- X_1\Psi_{1, k}, \ldots, \Gamma_{n,k}- X_n\Psi_{n, k}, 1- X_0\prod_{i=1}^{n}\prod_{j=1}^{k}\Psi_{i,j}\right).\]
For $i=1, \ldots, n$
\[\deg (\Gamma_{i,k}- X_i\Psi_{i, k})\leq k^{n-i}+1\leq C_3(n, d)k^{n-i}\]
and
\[h(\Gamma_{i,k}- X_i\Psi_{i, k})\leq h(F_{i}^{(k)})+\log 2 \leq C_4(n, d, h)k^{n-i+2}.\]
Now 
\[\deg \left(X_0\prod_{i=1}^{n}\prod_{j=1}^{k}\Psi_{i,j}\right)\leq 1+ \sum_{i=1}^{n}\sum_{j=1}^{k} j^{n-i}\leq C_5(n, d, h) k^{n(n+1)/2}.\]
By Lemma \ref{lem2d} 
\begin{align*}
&h\left(X_0\prod_{i=1}^{n}\prod_{j=1}^{k}\Psi_{i,j}\right) = h\left(\prod_{i=1}^{n}\prod_{j=1}^{k}\Psi_{i,j}\right) \\
& \leq \sum_{i=1}^{n}\sum_{j=1}^{k} h(\Psi_{i,j})+ \log(n+1) \left(\sum_{i =1}^n\sum_{j=1}^k \deg \Psi_{i,j}\right)\\
&\leq \sum_{i=1}^{n}\sum_{j=1}^{k}  j^{n-i+2} + C_5\log(n+1) nk^{n(n+1)/2}\leq C_6(n, d, h) k^{n(n+5)/2}.
\end{align*}
We apply Lemma \ref{lem4} with $n+1$ polynomials and $n+1$ variables,
\begin{align*}
\log \frak{B}_2 &\ll_{n,d,h}(k^{n(n+1)/2})^{3(n+1)+1}k^{n(n+5)/2} +(k^{n(n+1)/2})^{3(n+1)+2}\\
& \ll_{d, h, n}k^{n(3n^2+8n+9)/2}.
\end{align*}
Again by Bezout's theorem, 
\[T_{k} \leq 2k^{n(n+1)/2}\prod_{i=1}^{n} k^{n-i}\ll_{n, d, h} k^{n^2}.\]
This completes the proof of theorem. \qed

\subsection{Proof of Theorem \ref{th2}}
Let $p \in \mathbb{Z}$ be a prime and let ${\bf L} =(L_1, \ldots, L_s)\in \mathbb{Z}[{\bf X}]$  be a system of polynomials of degree at most $D$ and height at most $H$. We denote  by $V$ the subvariety of the affine space $\mathbb{A}_{\mathbb{Q}}^n$ defined by this system of polynomials.  We also denote the reduction modulo $p$ of the iteration ${\bf F}^{(m)}$ and V by ${\bf F}_p^{(m)}$ and $V_p$, respectively. Here we fix an initial point ${\bf w} \in \overline{\mathbb{F}}_p^n$ and let
\[A= \# \left\{ m \in \{0, 1, \ldots, \ell-1\} \mid {\bf F}_p^{(m)}({\bf w}) \in V_p(\overline{\mathbb{F}}_p)\right\}.\]
Suppose that 
\begin{equation}\label{eq12}
A> \epsilon \ell \geq 2.
\end{equation}
Take $\gamma \leq 2\ell/(A-1)$ and let $B$ be number of $m \in \{0, 1, \ldots, \ell-1\} $ with
\begin{equation}\label{eq13}
{\bf F}_p^{(m)}({\bf w}) \in V_p \quad \mbox{and}\quad  {\bf F}_p^{(m+ \gamma)}({\bf w}) =  {\bf F}_p^{(\gamma)}\left({\bf F}_p^{(m)}({\bf w}) \right) \in V_p.
\end{equation}
By Lemma \ref{lem4d}, 
\begin{equation}\label{eq14}
B \geq \frac{(A-1)^2}{4\ell}\gg \epsilon^2 \ell
\end{equation}
and hence we have $\gamma \ll 1/\epsilon.$ 

Since the iterations generically escape $V$, the set $\{{\bf u}\in V \mid {\bf F}^{(\gamma)}({\bf u}) \in V\}$ is finite and this set is defined by the following $2s$ equations
\begin{equation}\label{eq14a}
L_{j}({\bf X}) = L_{j}\left({\bf F}^{(\gamma)}\right)({\bf X}) = 0, \quad j= 1, \ldots, s.
\end{equation}
By Lemma \ref{lem2d} and \ref{lem2a}, we have 
\[\deg L_{j}\left({\bf F}^{(\gamma)}\right)\leq D n\gamma^{n-1}\]
and from B\'ezout's theorem   
\begin{align}\label{al15}
\begin{split}
\#\{{\bf u}\in V \mid {\bf F}^{(\gamma)}({\bf u}) \in V\} \leq D^s (D n\gamma^{n-1})^s\ll \frac{nD^{2s}}{\epsilon^{(n-1)s}}.
\end{split}
\end{align}

From Lemma \ref{lem2b}, we have 
\[h \left( F_{i}^{(\gamma)}\right) \ll_{d, h, n} \gamma^{n-i+2}\]
and hence by Lemma \ref{lem2d},
\[h(L_{j}\left({\bf F}^{(\gamma)}\right))\leq H + D\gamma^{n+1}+ (3Dn\gamma^{n-1} +1)D\log(n+1)\ll_{d, h,H, n}Dn\gamma^{n+1}.\]
Here the degree and height of $2s$ polynomials in \eqref{eq14a} are bounded by $Dn\gamma^{n-1}$ and 
$Dn\gamma^{n+1}$ respectively.
By Lemma \ref{lem4}, there is a positive integer $\mathfrak{D}$ with 
\begin{align*}\log  \mathfrak{D} &\leq (11n +4) (Dn \gamma^{n-1})^{3n+1}(Dn \gamma^{n+1}) \\
&+ (55n +99)\log((2n +5)s) (Dn \gamma^{n-1})^{3n+2}\leq \frac{C_7}{\epsilon^{n(3n-1)}}
\end{align*}
such that if $p\nmid \mathfrak{D}$, then 
\[\#\{{\bf u}\in V \mid {\bf F}^{(\gamma)}({\bf u}) \in V\} = \# \{{\bf u}\in V_p \mid {\bf F}_p^{(\gamma)}({\bf u}) \in V_p\}.\]

Since $S({\bf w})\geq \ell$, the points ${\bf F}_p^{(m)}({\bf w}), \, m=0, \ldots, \ell-1$ are pairwise distinct. Hence, 
\[B \leq \# \{{\bf u}\in V_p \mid {\bf F}_p^{(\gamma)}({\bf u}) \in V_p\}.\] 
From \eqref{eq14} and \eqref{al15}, we have
\[\epsilon^2 \ell \leq \frac{n^sD^{2s}}{\epsilon^{(n-1)s}}.\]
This is a contradiction as $\ell > \frac{n^sD^{2s}}{\epsilon^{(n-1)s+2}}$. Thus, $A\leq \epsilon \ell$ and this completes the proof of theorem \qed.

\subsection{Proof of Theorem \ref{th3}}
Set \[\beta=  \floor*{\frac{2L}{\epsilon}}+1;\]
thus $\ell \geq \beta.$
For each set $B \subseteq \{0, \ldots, \beta-1\}$ of cardinality $\# B = L+1$, we consider the system of equations 
\begin{equation}\label{eq16}
L_{j}\left({\bf F}^{(k)}\right)= 0, \quad k\in  B, j= 1, \ldots, s.
\end{equation}
Since $k\in B$, we have $k\leq \beta-1$. By Lemma \ref{lem2d} and \ref{lem2a},
\[\deg L_{j}\left({\bf F}^{(k)}\right) \leq k^{n-1}Dn \leq (\beta-1)^{n-1}Dn.\]
Again, by Lemma \ref{lem2d} and Lemma \ref{lem2b}, we have 
\begin{align*}
h \left( L_{j} \left({\bf F}^{(k)} \right) \right) &\leq  H + Dk^{n+1} +(3n k^{n-1} +1)D \log(n+1)\ll_{H, n}D(\beta-1)^{n+1}.
\end{align*}
Since the intersection of orbits of ${\bf F}$ with $V$ is $L$-uniformly bounded and $k\in B$, the system of equations in \eqref{eq16} has no common solution ${\bf w}\in \overline{\mathbb{Q}}^n$. By Lemma \ref{lem4}, there exists $\mathfrak{C}_{B} \in \mathbb{N}$ with 
\begin{align*}\log  \mathfrak{C}_{B}&\leq (11n +4) (Dn (\beta-1)^{n-1})^{3n+1}D(\beta-1)^{n+1} \\
&+ (55n +99)\log((2n +5)s) (Dn (\beta-1)^{n-1})^{3n+2}\\
& \leq C_8(Dn (\beta-1)^{n-1})^{3n+2}(\beta-1)^{n+1}.
\end{align*}
such that if $p$ is a prime and $p\nmid \mathfrak{C}_{B}$, then the reduction modulo $p$ of the system of equations \eqref{eq16} has no solutions in $\overline{\mathbb{F}}_p^n$.

Now set 
\[ \mathfrak{C} = \prod_{\substack{B\subseteq \{0, \ldots, \beta-1\}\\ \# B = L+1}} \mathfrak{C}_{B}\]
and hence
\begin{align}
\begin{split}
\log \mathfrak{C} &\ll_{D, d, h, H, n, L} \binom{\beta}{L+1}(Dn (\beta-1)^{n-1})^{3n+2}(\beta-1)^{n+1} \\
&\leq \frac{C_9(D, d, h, H, n, L, s)}{ \epsilon^{(n-1)(3n+2)+(n+L+2)}}.
\end{split}
\end{align}
Let $p$ be a prime with $p\nmid \mathfrak{C}$. Suppose that for some ${\bf u} \in {\overline{\mathbb{F}}}_p^n$ there are at least $\epsilon \ell$ values of $n\in \{0, \ldots, \ell-1\}$ with ${\bf F}_p^{(n)}({\bf u}) \in V_p$. Since $\ell \geq \beta$, there is a non negative integer $i\leq \floor*{\ell/\beta}$ such that there are at least 
\[\frac{\epsilon \ell}{\floor*{\ell/\beta}+1}\geq \frac{\epsilon \beta}{2} > L\]
values of $n\in \{i\beta, \ldots, (i+1)\beta-1\}$ with ${\bf F}_p^{(n)}({\bf u}) \in V_p$. Now consider $L+1$ values 
\[i\beta< i\beta+ \delta_1<\cdots < i\beta+\delta_{L+1} < (i+1)\beta.\]
Then for $j=1, \ldots, s$ and $t =1, \ldots, L+1$, 
\[L_{j}\left({\bf F_p}^{(\delta_{t})}\left({\bf F_p}^{(i\beta)}\right)\right) =0.\]
Setting ${\bf w}= {\bf F_p}^{(i\beta)} \in {\overline{\mathbb{F}}}_p^n$, then for all $j, t$
\[L_{j}\left({\bf F_p}^{(\delta_{t})}\left({\bf w}\right)\right) =0.\]
This implies that $p\nmid \mathfrak{C}_{B}$ with $B = \{\delta_1, \ldots, \delta_{L+1}\}$ which is a contradiction. This completes the proof of theorem \qed.

\end{document}